\documentclass[11pt]{article}

 \usepackage{amsfonts,amssymb,amsmath,mathrsfs}

\usepackage{color,latexsym,amsfonts}
\newcommand{\qed}{\hfill\rule{2mm}{3mm}\vspace{4mm}}

\numberwithin{equation}{section}

 \newtheorem{theorem}{Theorem}[section]
 \newtheorem{lemma}[theorem]{Lemma}
 \newtheorem{definition}[theorem]{Definition}
 \newtheorem{remark}{Remark}
 \newtheorem{assumption}[theorem]{Assumption}
 \newtheorem{corollary}[theorem]{Corollary}
 
 \newtheorem{example}[theorem]{Example}

 \def\<{\langle}\def\>{\rangle}

 \def\beqnn{\begin{eqnarray*}}\def\eeqnn{\end{eqnarray*}}
 \def\<{\langle}\def\>{\rangle}

\def\beqlb{\begin{eqnarray}}\def\eeqlb{\end{eqnarray}}
 \def\qed{\hfill$\Box$\medskip}

\begin{document}

\bigskip\bigskip
\noindent{\Large\bf Asymptotic behaviour of heavy-tailed branching processes  in random environments }\footnote{
\noindent  Supported by NSFC (NO.~11531001).
}

\noindent{%\normalsize\sf
Wenming Hong\footnote{ School of Mathematical Sciences
\& Laboratory of Mathematics and Complex Systems, Beijing Normal
University, Beijing 100875, P.R. China. Email: wmhong@bnu.edu.cn} ~ Xiaoyue Zhang\footnote{ School of Mathematical Sciences
\& Laboratory of Mathematics and Complex Systems, Beijing Normal
University, Beijing 100875, P.R. China. Email: zhangxiaoyue@mail.bnu.edu.cn}

\noindent{%\normalsize\sf
(Beijing Normal University)
}

}

\begin{center}
\begin{minipage}{12cm}
\begin{center}\textbf{Abstract}\end{center}
\footnotesize
Consider a heavy-tailed branching process (denoted by $Z_{n}$)  in random environments, under the condition which infers that $\mathbb{E}\log m(\xi_{0})=\infty$. We show that  (1) there exists no proper $c_{n}$ such that $\{Z_{n}/c_{n}\}$ has a proper, non-degenerate limit; (2) normalized by a sequence of functions, a proper limit can be obtained, i.e., $y_{n}\left(\bar{\xi},Z_{n}(\bar{\xi})\right)$ converges almost surely to a random variable $Y(\bar{\xi})$, where $Y\in(0,1)~\eta$-a.s.; (3) finally, we give a necessary and sufficient conditions for the almost sure convergence of $\left\{\frac{U(\bar{\xi},Z_{n}(\bar{\xi}))}{c_n(\bar{\xi})}\right\}$, where $U(\bar{\xi})$ is a slowly varying function that may depends on $\bar{\xi}$.
\bigskip

\mbox{}\textbf{Keywords:}\quad branching process, random environment, heavy-tailed,  regular,  irregular,  martingales. \\
\mbox{}\textbf{Mathematics Subject Classification}:  Primary 60J80;
secondary 60F10.

\end{minipage}
\end{center}

\section{ Introduction}

Let $\{Z_{n}\}$ be a Galton-Watson branching process with $Z_{0}=1$ and governed by the family size probability generating function $f(s)=\sum_{j=0}^{\infty}p_{j}s^{j}$, where $p_{1}\neq1$. Let $m=\sum_{j=0}^{\infty}jp_{j}=f^{'}(1-)$ denotes the mean of the offspring distribution.

% It is well known that in the subcritical and critical case ($m\leqslant 1$), the population almost surely dies out. In the supercritical case ($m>1$), $Z_{n}\to \infty$ with probability $1-q$, where $q$ is the smaller root of $f(s)=s$.

Martingale convergence of branching processes have been investigated extensively. Kesten and Stigum ($\cite{KS66}$) showed that the limit of the martingale
$\left\{\frac{Z_{n}}{m^{n}}\right\}$ is proper if and only if $EZ_1\log Z_1 < \infty$.
%($\sum_{j=1}^{\infty}p_{j}j\log j<\infty$).
 After that, if only the  condition $EZ_1<\infty$ is fulfilled,  Seneta ($\cite{SE69}$) showed that if we use $f_{n}(s)$ to denote the probability generating function of $Z_{n}$, $k_{n}(s)=-\log f_{n}(e^{-s})$, $h_{n}(s)$ is the inverse function of $k_{n}(s)$, then for every $s\in (0,-\log q)$, $Z_{n}h_{n}(s)$ converges in distribution to a proper, non-degenerate law. Heyde ($\cite{HC70}$) strengthened this result to almost sure convergence, using a martingale argument.

When $m=\infty$, the situation is more complicated. In this case, Seneta ($\cite{SE69}$) showed that it is never possible to find $\{c_{n}\}$ such that $\left\{\frac{Z_{n}}{c_{n}}\right\}$ converges in distribution to a proper, non-degenerate law. Darling ($\cite{DD70}$) and Seneta ($\cite{SE73}$) gave sufficient conditions for the existence of a sequence $\{c_{n}\}$ such that $\left\{\frac{\log (Z_{n}+1)}{c_{n}}\right\}$ converges in distribution to a non-degenerate law. Schuh and Barbour ($\cite{SB77}$) showed that branching process with infinite mean can be classified as regular or irregular according to the property that whether there exists a sequence of constants $\{c_{n}\}$ such that $P\left(0<\lim_{n\to\infty}\frac{Z_{n}}{c_{n}}<\infty\right)>0$. In that paper, they derived necessary and sufficient conditions for the almost sure convergence of $\frac{U(Z_{n})}{c_{n}}$, where $U$ is a slow varying function, moreover, the distribution function of the limit satisfies a $Poincar\acute{e}$ functional equation.

When this model is extended to a random environment, the corresponding martingale convergence results have been proved by Tanny($\cite{TD78}$, $\cite{TD88}$).  The Kesten-Stigum type theorem was proved in Tanny ($\cite{TD88}$),  $\lim_{n\to\infty}\frac{Z_{n}}{\pi_n}=W$ $w.p.1$, where $W$ is proper and non-degenerate if and only if $E(Z_1\log^{+} Z_1/m(\xi_{0})) < \infty$ when the environmental sequence $\bar{\xi}=(\xi_{0},\xi_{1},\cdots)$ is i.i.d. (Theorem 2,  $\cite{TD88}$), where $m(\xi_{i})$ is the expected number of offspring of particle conditioned on the environment $\xi_{i}$, and $\pi_n:=\Pi_{i=0}^{n-1}m(\xi_{i})$.
The Seneta-Heyde type theorem  was considered in Tanny ($\cite{TD78}$) if the environmental sequence $\bar{\xi}=(\xi_{0},\xi_{1},\cdots)$ is stationary and ergodic and satisfies $\mathbb{E}|\log m(\xi_{0})|<\infty$ , then there exists a sequence of random variables $c_{n}(\bar{\xi})$, depending only on the environment sequence $\bar{\xi}$ such that $\lim_{n\to\infty}\frac{Z_{n}}{c_n}=W$ $w.p.1$ and $W$ is proper and non-degenerate, i.e., $\mathbb{P}(0<W<\infty|\bar{\xi})=1-q(\bar{\xi})$, where $q(\bar{\xi})$ is the extinction probability conditioned on $\bar{\xi}$.

In the present paper, we are interested in
the case $\mathbb{E}|\log m(\xi_{0})|=\infty$.  We investigate the asymptotic behaviors of branching processes in random environments under the condition {\bf(A2)} (which  infers $\mathbb{E}|\log m(\xi_{0})|=\infty$). Part of the results in Schuh and Barbour ($\cite{SB77}$) will be extended to this random environment situation, in particular, (1) we show that  for $a.s.$ $\bar{\xi}$, there exists no $\{c_{n}(\bar{\xi})\}$ such that $\left\{\frac{Z_{n}(\bar{\xi})}{c_{n}(\bar{\xi})}\right\}$ converges to a proper random variable;
(2)If $Z_{n}$ can be normalized by a sequence of functions, i.e., let $y_{n}(\bar{\xi},x)=f_{\xi_{0}}\left(\cdots\left(f_{\xi_{n-1}}\left(e^{-\frac{1}{x}}\right)\right)\cdots\right)$, then $y_{n}(\bar{\xi},Z_{n}(\bar{\xi}))$ converges almost surely to a  proper and non-degenerate random variable $Y(\bar{\xi})$, where $Y\in(0,1)~\eta$-a.s.;
(3) we give the necessary and sufficient conditions for the almost sure convergence of $\left\{\frac{U(\bar{\xi},Z_{n}(\bar{\xi}))}{c_n(\bar{\xi})}\right\}$, where $U(\bar{\xi})$ is a slowly varying function that may depends on $\bar{\xi}$.

 \section{Description of the model and main results}

Let $\bar{\xi}=\{\xi_{n}:n\in \mathbb{Z}\}$ be a sequence of independent and identically distributed probability distributions on nonnegative integers, where
\beqnn
\xi_{n}=\left\{\xi_{n}^{(0)},\xi_{n}^{(1)},\cdots\right\},~~~~~~~~\xi_{n}^{(i)}\geqslant0,~~~~ ~~~~\sum_{i=0}^{\infty}\xi_{n}^{(i)}=1.
\eeqnn
The law of the environment $\bar{\xi}$ is given by $\eta$.

Let $Z_{0}=1$, $Z_{n}$ be the sum of $Z_{n-1}$ independent random variables, each of which has distribution $\xi_{n-1}$. Then the sequence of random variables $Z_{0},Z_{1},\cdots$ is called a branching process in the random environment $\bar{\xi}$. We use $P_{\bar{\xi}}$ to denote the probability when the environment $\bar{\xi}$ is fixed. As usual, $P_{\bar{\xi}}$ is called quenched law. The total probability $\mathbb{P}$, which is usually called annealed law, is given by
\beqnn
\mathbb{P}(\cdot):= \int P_{\bar{\xi}}(\cdot)\eta(d\bar{\xi}).
\eeqnn

\begin{assumption}\label{Ass}

{\bf(A1)} $\eta(\xi_{0}^{(0)}=0)=1.$

\end{assumption}

\begin{remark}
\noindent {\bf(A1)} ensures that each particle produces at least one particle, then this is an increasing branching process in the random environment, i.e., the extinction probability $q(\bar{\xi})=0$. We propose this assumption to simplify our statement, but in fact this assumption can be removed by using $Theorem~2$ in $\cite{TD78}$ to get the main result in our passage on the non-extinction event.
 \qed

\end{remark}

{\noindent \bf Some notations:}

$\bullet$   $m(\xi_{0})=E_{{\xi}_0}(Z_1):=\sum_{y=0}^{\infty}y\xi_{0}^{(y)}$;

$\bullet$ $k_{\xi_{i}}(s)=-\log f_{\xi_{i}}(e^{-s})$; ~~  $h_{\xi_{i}}(s)=-\log f_{\xi_{i}}^{(-1)}(e^{-s})$,  $0<s<\infty$;

$\bullet$ $k_{n}(\bar{\xi},s):=k_{\xi_{0}}(k_{\xi_{1}}(\cdots(k_{\xi_{n-1}}(s))\cdots)=-\log f_{\xi_{0}}
\left(f_{\xi_{1}}\left(\cdots\left(f_{\xi_{n-1}}(e^{-s})\right)\cdots\right)\right)$,

$\bullet$ $h_{n}(\bar{\xi},s)=h_{\xi_{n-1}}(\cdots(h_{\xi_{0}}(s))\cdots)=-\log f_{\xi_{n-1}}^{(-1)}\left(\cdots\left(f_{\xi_{0}}^{(-1)}\left(e^{-s}\right)\right)\cdots\right)$;

$\bullet$ $\theta$ is the shift operator, for any $\bar{\xi}=\{\xi_{0},\xi_{1},\cdots\}$,  $\theta\bar{\xi}:=\{\xi_{1},\xi_{2},\cdots\}$;

$\bullet$ $d\left(\bar{\xi},s\right):=\lim_{n\to\infty}\frac{h_{n+1}\left(\bar{\xi},s\right)}{h_{n}\left(\theta\bar{\xi},s\right)}=
\lim_{n\to\infty}\frac{h_{\xi_{n}}\left(\cdots\left(h_{\xi_{0}}(s)\right)\cdots\right)}{h_{\xi_{n}}\left(\cdots(h_{\xi_{1}}(s))\cdots\right)}$.

\begin{assumption}\label{Ass}

{\bf (A2)} $\eta(D)=1$, where $D=\left\{\bar{\xi}:\text{for any } s\in(0,\infty),~d(\bar{\xi},s)=0\right\}$.
\end{assumption}

\begin{remark}

\noindent  Tanny ($Lemma~2.4$, \cite{TD78})  proved  that if $\mathbb{E}|\log m(\xi_{0})|<\infty$, then $0<d(\bar{\xi},s)\leqslant1~~w.p.1$.  Thus under the assumption {\bf(A2)} we know $\mathbb{E}|\log m(\xi_{0})|=\infty$. Actually, we conjecture {\bf(A2)} is equivalence with $\mathbb{E}|\log m(\xi_{0})|=\infty$, but unfortunately we have not proved it yet.

An example is given in Example \ref{exm}, where the Assumption {\bf(A1)} and {\bf(A2)} are fulfilled. \qed

\end{remark}

\begin{remark}

\noindent The heavy-tailed behavior have been considered in branching random walk, see for example (\cite{A13}) and (\cite{V87}), etc.. \qed
\end{remark}

{\bf 1 { No proper limit exists} }

If $\mathbb{E}|\log m(\xi_{0})|<\infty$, Tanny ($\cite{TD78}$) proved that there exists a sequence of random variables $c_{n}(\bar{\xi})$, depending only on the environment sequence $\bar{\xi}$ such that $\lim_{n\to\infty}\frac{Z_{n}}{c_n}=W$ $w.p.1$ and $W$ is proper and non-degenerate, i.e., $\mathbb{P}\left(0<W<\infty|\bar{\xi}\right)=1-q(\bar{\xi})$, where $q(\bar{\xi})$ is the extinction probability conditioned on $\bar{\xi}$. The key step of the proof is in Tanny ($Lemma~2.4$, \cite{TD78}), where showed  that if $\mathbb{E}|\log m(\xi_{0})|<\infty$, then $0<d(\bar{\xi},s)\leqslant1~~w.p.1$.

We are interested in the other situation when $d(\bar{\xi},s)=0 ~~w.p.1$, $0<s<\infty$, i.e.,  Assumption \ref{Ass} (where (A2) infer that $\mathbb{E}|\log m(\xi_{0})|=\infty$), we will see that for a.e. $\bar{\xi}$, no $c_{n}(\bar{\xi})$ exists such that $Z_{n}(\bar{\xi})/c_{n}(\bar{\xi})$ has a proper and non-degenerate limit. At first, we show that for any $s\in(0,\infty)$, $ h_{n}(\bar{\xi},s)$ is not the suitable norming for $Z_{n}(\bar{\xi})$ as the following,

\begin{theorem}\label{t1}  For any $s\in(0,\infty)$, $Z_{n}(\bar{\xi})h_{n}(\bar{\xi},s)$ converges to $W(\bar{\xi},s)~w.p.1$. If $\eta(D)>0$ then $P_{\bar{\xi}}\left(W
(\bar{\xi},s)=\infty\right)>0,~P_{\bar{\xi}}\left(W(\bar{\xi},s)=0\right)>0,~\eta\text{-a.e..}$
\end{theorem}

\begin{remark}
Note that the condition in Theorem \ref{t1} is  weaker than condition (A2). We conjecture that $\eta(D)=0 ~\mbox{or} \  1$. \qed
\end{remark}

Based on this facts, it is necessary to classify $s\in(0,\infty)$ as  two different  types of points from the following definition.
\begin{definition}
A point $s\in(0,\infty)$ is called $\bar{\xi}$-regular if $P_{\bar{\xi}}\left(W(\bar{\xi},s)\in\{0,\infty\}\right)=1$,
 and $\bar{\xi}$-irregular otherwise.
\end{definition}

\begin{definition}\label{prir}
The branching process $\left\{Z_{n}(\bar{\xi})\right\}$ is called $\bar{\xi}$-regular if all $0<s<\infty$ are $\bar{\xi}$-regular and $\bar{\xi}$-irregular otherwise.
\end{definition}

We have the following 0-1 law.
\begin{theorem}\label{tprir}
Let $A=\left\{\bar{\xi}:Z_{n}(\bar{\xi})~\text{is}~\bar{\xi}\text{-regular}\right\}$, then $\eta(A)=0~\text{or}~1$.
\end{theorem}

In section \ref{srir},  we will discuss the limit behavior of $Z_{n}(\bar{\xi})h_{n}(\bar{\xi},s)$ in details, and finally conclude that no $c_{n}(\bar{\xi})$ exists such that $Z_{n}(\bar{\xi})/c_{n}(\bar{\xi})$ has a proper and non-degenerate limit.

\

{\bf 2 {Normalized by a sequence of functions} }

Since for $\eta$-a.e. $\bar{\xi}$, $\lim_{n}Z_{n}(\bar{\xi})/c_{n}(\bar{\xi})$ is never a proper, non-degenerate random variable, we now consider other possibilities for normalizing $Z_{n}(\bar{\xi})$.

For $0\leqslant x<\infty$, let
\beqnn
y_{n}\left(\bar{\xi},x\right)=f_{\xi_{0}}\left(\cdots\left(f_{\xi_{n-1}}\left(e^{-\frac{1}{x}}\right)\right)\cdots\right), ~\left(y_{n}(\bar{\xi},0)=f_{n}(\bar{\xi},0)\right).
\eeqnn

\begin{theorem}\label{yu}
 $y_{n}\left(\bar{\xi},Z_{n}(\bar{\xi})\right)$ converges almost surely to a random variable $Y(\bar{\xi})$, where $Y\in(0,1)~\eta$-a.s..

If $s_{r}$ is a $\bar{\xi}$-regular point and $x_{r}=e^{-s_{r}}$, then $P_{\bar{\xi}}\left(Y(\bar{\xi})\leqslant x_{r}\right)=x_{r}$ and $P_{\bar{\xi}}\left(Y(\bar{\xi})= x_{r}\right)=0$. In particular, if $\{Z_{n}\}$ is a regular branching process, then $Y$ is uniformly distributed on $(0,1)$.

\end{theorem}

{\bf 3 {Normalized by an increasing slowly varying function} }

\begin{theorem}\label{t2} For $\eta$-a.e.$~\bar{\xi}$, let $U(\bar{\xi},x):[0,\infty)\to[0,\infty)$ be an increasing slowly varying function with $U(\bar{\xi},0)=0,~\lim_{x\to\infty}U(\bar{\xi},x)=\infty$, and $\{c_{n}(\bar{\xi})\}$ a sequence of positive constants. Then under Assumption \ref{Ass} we have:

\noindent $\left( 1 \right)$ For $\eta$-a.e.$~\bar{\xi}$, if
\begin{eqnarray}\label{exi}
H(\bar{\xi},s):=\displaystyle\lim_{n}\left(U(\bar{\xi},1/h_{n}(\bar{\xi},s))/c_{n}(\bar{\xi})\right)
\end{eqnarray}
exists for all but at most countably many $s\in(0,\infty)$, then for $\eta$-a.e.$~\bar{\xi}$, $U\left(\bar{\xi},Z_{n}(\bar{\xi})\right)/c_{n}(\bar{\xi})$ converges to $H\left(\bar{\xi},T(\bar{\xi})\right)$  $\left(T(\bar{\xi})=\sup\left\{s|0<s<\infty\text{ and } W(\bar{\xi},s)<1\right\}\right)$ almost surely.

\noindent $\left( 2 \right)$ On the other hand, if for $\eta$-a.e.$~\bar{\xi}$, $\left\{U(\bar{\xi},Z_{n}(\bar{\xi}))/c_{n}(\bar{\xi})\right\}$ converges in distribution to a distribution function $F_{\bar{\xi}}$, and define
\begin{eqnarray}
G\left(\bar{\xi},x\right)=\inf\left\{y~|~0\leqslant y<\infty~\text{and}~F_{\bar{\xi}}(y)\geqslant x\right\},~~~0\leqslant x<\infty.
\end{eqnarray}
Then for $\eta$-a.e.$~\bar{\xi}$,
\begin{eqnarray}\label{2.3}
\lim_{n}\left(U(\bar{\xi},1/h_{n}(\bar{\xi},s))/c_{n}(\bar{\xi})\right)=G(\bar{\xi},e^{-s})
\end{eqnarray}
exists for all the points $s\in (0,\infty)$ such that $G(\bar{\xi})$ is continues at $e^{-s}\big($Since $G(\bar{\xi})$ is an increasing function, these are all but at most countably many $s\in (0,\infty)\big)$.

\noindent $\left( 3 \right)$ Furthermore, under the condition of (2), if for $\eta$-a.e.$~\bar{\xi}$, $U(\bar{\xi},x)=U(\theta\bar{\xi},x)$ and for $\eta$-a.e.$~\bar{\xi}$, $F_{\bar{\xi}}$ satisfies for any $0<x<\infty$,
\begin{eqnarray}
0<F_{\bar{\xi}}(x)<1,&~~~\displaystyle\lim_{x\to 0}F_{\bar{\xi}}(x)=0,~~~\lim_{x\to \infty}F_{\bar{\xi}}(x)=1.
\end{eqnarray}
Then,
\begin{eqnarray}\label{di}
\lim_{n}\frac{c_{n-1}(\theta\bar{\xi})}{c_{n}(\bar{\xi})}=\alpha(\bar{\xi})>0
\end{eqnarray}
exists, and
\begin{eqnarray}\label{gth}
G\left(\bar{\xi},e^{-s}\right)/G\left(\theta\bar{\xi},e^{-h_{\xi_{0}}(s)}\right)=\alpha\left(\bar{\xi}\right)~~~~~\text{for } s\in(0,\infty).
\end{eqnarray}
What's more, for $\eta$-a.e.$~\bar{\xi}$, the distribution function $F_{\bar{\xi}}$ and $F_{\theta\bar{\xi}}$ satisfy the functional equation
\begin{eqnarray}\label{xidi}
F_{\bar{\xi}}\left(\alpha(\bar{\xi})u\right)=f_{\xi_{0}}\left(F_{\theta\bar{\xi}}(u)\right),~~0\leqslant u<\infty,
\end{eqnarray}
where $\alpha(\bar{\xi})$ is as in $(\ref{di})$.

\end{theorem}

The  paper is arranged as the following. All the above results will be proved in section {\ref{ms}}.  Based on  Theorem {\ref{t1}} (which will be proved in section {\ref{ms1}}), we will give the classification for $s\in (0,\infty)$ as the regular and irregular point in section {\ref{srir}}, and  some facts for the regular and irregular point will be pointed out, but the proof  will omit as it is similar as those in  Schuh and Barbour ($\cite{SB77}$). In section {\ref{ms3}}, a proper limit will be obtained when $Z_n$ is  normalized by a sequence of functions, i.e., Theorem {\ref{yu}} will be proved. In section {\ref{ms4}}, we will discuss the necessary and sufficient conditions for the almost sure convergence of $\left\{\frac{U(\bar{\xi},Z_{n}(\bar{\xi}))}{c_n(\bar{\xi})}\right\}$, where $U(\bar{\xi})$ is a slowly varying function that may depends on $\bar{\xi}$, i.e.,
Theorem {\ref{t2}} will be proved.  In section \ref{sms}, we will give a sufficient  criteria for regular process, and finally we give an Example \ref{exm} to illustrate our results in this paper.

\section{ Proofs}\label{ms}

\subsection{Proof of Theorem $\ref{t1}$}\label{ms1}

\begin{proof}
\noindent $\left( \text{i} \right)$ Denote by $\mathscr{F}_{n}(\bar{\xi})$ the $\sigma$-field generated by $Z_{0},Z_{1}, \cdots, Z_{n}$ and $\bar{\xi}$, let
\begin{eqnarray}
X_{n}(\bar{\xi},s):=e^{-Z_{n}(\bar{\xi})h_{n}(\bar{\xi},s)}.
\end{eqnarray}
Then it is easy to check that $\left\{X_{n}(\bar{\xi},s),\mathscr{F}_{n}(\bar{\xi})\right\}_{n=0}^{\infty}$ is a martingale bounded between 0 and 1, by the Martingale Convergence Theorem,
\begin{eqnarray*}
X(\bar{\xi},s):=\lim_{n\to\infty}X_{n}(\bar{\xi},s)~~~~~\text{exists   w.p.1}.
\end{eqnarray*}
By the property of martingale,
\begin{eqnarray*}
\mathbb{E}(X_{n}(\bar{\xi},s)|\bar{\xi})=\mathbb{E}(X_{0}(\bar{\xi},s)|\bar{\xi})=e^{-s}
\end{eqnarray*}
\begin{eqnarray}
\mathbb{E}(X(\bar{\xi},s)|\bar{\xi})=\lim_{n\to\infty}\mathbb{E}(X_{n}(\bar{\xi},s)|\bar{\xi})=e^{-s}
\end{eqnarray}
Therefore for any $s\in (0,\infty),~ Z_n(\bar{\xi})h_{n}(\bar{\xi},s)$ converges to $W(\bar{\xi},s):=-\log X(\bar{\xi},s)~~w.p.1.$

\noindent $\left( \text{ii} \right)$
Let $W_{n}(\bar{\xi},s)=Z_{n}(\bar{\xi})h_{n}(\bar{\xi},s),$ then
$X_{n}(\bar{\xi},s)^u=e^{-uW_{n}(\bar{\xi},s)}=e^{-uZ_{n}(\bar{\xi})h_n(\bar{\xi},s)}$,
\begin{eqnarray}\label{pass}
 E_{\bar{\xi}}\left[ X_{n}(\bar{\xi},s)^u\right]&=& E_{\bar{\xi}}\left[e^{-uZ_{n}(\bar{\xi})h_n(\bar{\xi},s)}\right]\nonumber\\
&=& f_{\xi_{0}}\left(\cdots\left(f_{\xi_{n-1}}\left(e^{-uh_{n}(\bar{\xi},s)}\right)\right)\cdots\right)\nonumber\\
&=& f_{\xi_{0}}\left(f_{\xi_{1}}\left(\cdots\left(f_{\xi_{n-1}}\left(e^{-uh_{n-1}(\theta\bar{\xi},h_{\xi_{0}}(s))}\right)\right)\cdots\right)\right)\nonumber\\
&=& f_{\xi_{0}}\left(E_{\theta\bar{\xi}}\left(X_{n-1}(\theta\bar{\xi},h_{\xi_{0}}(s))^u\right)\right)
\end{eqnarray}
Let $\chi(u;\bar{\xi},s)=E_{\bar{\xi}}(X(\bar{\xi},s)^u)$, then $n$ goes to infinity in (\ref{pass}) yields
\begin{eqnarray}\label{x1}
\chi(u;\bar{\xi},s)=f_{\xi_{0}}\left(\chi(u;\theta\bar{\xi},h_{\xi_{0}}(s))\right).
\end{eqnarray}
It is easily seen that
\begin{eqnarray}\label{x2}
\lim_{u\downarrow 0}\chi(u;\bar{\xi},s)=\lim_{u\downarrow 0}E_{\bar{\xi}}\left(e^{-uW(\bar{\xi},s)}\right)=P_{\bar{\xi}}\left(W(\bar{\xi},s)<\infty\right).
\end{eqnarray}
Using the functional relation $(\ref{x1})$ with equation $(\ref{x2})$ yields:
\begin{eqnarray}\label{x3}
P_{\bar{\xi}}\left(W(\bar{\xi},s)<\infty\right)=f_{\xi_{0}}\left(P_{\theta\bar{\xi}}(W(\theta\bar{\xi},h_{\xi_{0}}(s))<\infty)\right).
\end{eqnarray}
Let
\beqnn
A=\left\{\bar{\xi}:\text{there exists s such that } P_{\bar{\xi}}(W(\bar{\xi},s)<\infty)=1\right\}.
\eeqnn
Combined with $(\ref{x3})$ and the property of $f_{\xi_{0}}(s)$, we have $\theta A=A$, $i.e.$ A is a $\theta$-invariant set, since $\theta$ is ergodic,
\begin{eqnarray}\label{01}
\eta(A)=0~~~\text{or}~~~1.
\end{eqnarray}
Then we only need to prove $\eta(A)\neq 1.$

Since
\begin{eqnarray}\label{x5}
E_{\bar{\xi}}\left[X_{n}(\bar{\xi},s)^u\right]=
f_{\xi_{0}}\left(f_{\xi_{1}}\left(\cdots\left(f_{\xi_{n-1}}\left(e^{-uh_{n-1}(\theta\bar{\xi},s)\frac{h_{n}(\bar{\xi},s)}{h_{n-1}(\theta\bar{\xi},s)}}\right)\right)\cdots\right)\right),
\end{eqnarray}
let $n\to\infty$, we have
\begin{eqnarray}\label{x6}
\chi(u;\bar{\xi},s)=f_{\xi_{0}}\left(\chi(ud(\bar{\xi},s);\theta\bar{\xi},s)\right).
\end{eqnarray}
Consequently for any $\bar{\xi}\in D,~0<s<\infty$, $u>0$, we can get that
\beqnn
\chi(u;\bar{\xi},s)=f_{\xi_{0}}\left(\chi(0;\theta\bar{\xi},s)\right).
\eeqnn
 Using the fact that
\beqnn
\lim_{u\uparrow\infty}\chi(u;\bar{\xi},s)=\lim_{u\uparrow \infty}E_{\bar{\xi}}\left(e^{-uW(\bar{\xi},s)}\right)=P_{\bar{\xi}}\left(W(\bar{\xi},s)=0\right),
\eeqnn
we have
\begin{eqnarray}\label{x4}
\text{for any } \bar{\xi}\in D,~~P_{\bar{\xi}}\left(W(\bar{\xi},s)=0\right)=f_{\xi_{0}}\left(\chi(0;\theta\bar{\xi},s)\right).
\end{eqnarray}
Note that $E_{\bar{\xi}}\left(e^{-W(\bar{\xi},s)}\right)=e^{-s}$, which implies $P_{\bar{\xi}}\left(W(\bar{\xi},s)=0\right)<1$. From (\ref{x4}) we have $\chi\left(0;\theta\bar{\xi},s\right)<1$, thus
\beqnn
\text{for any } \bar{\xi}\in D,~~P_{\theta\bar{\xi}}\left(W(\theta\bar{\xi},s)=\infty\right)>0.
\eeqnn
This means for any $\bar{\xi}\in D,~\theta\bar{\xi}\in A^{c}$. Hence
\beqnn
\eta(A)=1-\eta(A^{c}) \leqslant 1-\eta(\theta D)=1-\eta(D)<1,
\eeqnn
recall (\ref{01}), we have  $\eta(A)=0,~i.e.,$
\beqnn
\text{for any }s\in(0,\infty),~~~P_{\bar{\xi}}\left(W(\bar{\xi},s)=\infty\right)>0,~~~\eta\text{-a.e.}.
\eeqnn

\noindent $\left( \text{iii} \right)$
Let
\beqnn
B=\left\{\bar{\xi}:\text{there exists s such that } P_{\bar{\xi}}(W(\bar{\xi},s)=0)=0\right\}.
\eeqnn
Similar to (ii) we can get that
\begin{eqnarray}\label{x7}
P_{\bar{\xi}}\left(W(\bar{\xi},s)=0\right)=f_{\xi_{0}}\left(P_{\theta\bar{\xi}}(W(\theta\bar{\xi},h_{\xi_{0}}(s))=0)\right),
\end{eqnarray}
\begin{eqnarray}\label{x8}
\chi\left(\frac{u}{d(\bar{\xi},s)};\bar{\xi},s\right)=f_{\xi_{0}}\left(\chi(u;\theta\bar{\xi},s)\right).
\end{eqnarray}
From $(\ref{x7})$, we know that $\theta B=B$, then $\eta(B)=0 ~~or~~ 1$.

For any $\bar{\xi}\in D,~0<s<\infty$, from $(\ref{x8})$, we get for any $u>0$,
\beqnn
f_{\xi_{0}}\left(\chi(u;\theta\bar{\xi},s)\right)=\chi\left(\infty;\bar{\xi},s\right).
\eeqnn
Let $u$ goes to 0, we have
\begin{eqnarray}\label{x9}
f_{\xi_{0}}\left(P_{\theta\bar{\xi}}(W(\theta\bar{\xi},s)<\infty)\right)=\chi\left(\infty;\bar{\xi},s\right).
\end{eqnarray}
We note that $E_{\theta\bar{\xi}}\left(e^{-W(\theta\bar{\xi},s)}\right)=e^{-s}$ implies  $P_{\theta\bar{\xi}}(W(\theta\bar{\xi},s)<\infty)>0$. Combined with (\ref{x9}) we have $\chi(\infty;\bar{\xi},s)>0$, thus
\beqnn
\text{for any } \bar{\xi}\in D,~0<s<\infty,~~P_{\bar{\xi}}\left(W(\bar{\xi},s)=0\right)>0.
\eeqnn
This means for any $\bar{\xi}\in D,~\bar{\xi}\in B^{c}$. Then
\beqnn
\eta(B)\leqslant 1-\eta(D)<1.
\eeqnn
 Accordingly $\eta(B)=0,~i.e.,$
\beqnn
\text{for any }s\in(0,\infty),~~~P_{\bar{\xi}}\left(W(\bar{\xi},s)=0\right)>0,~~~\eta\text{-a.e.}.
\eeqnn
\qed
\end{proof}

 \subsection{$\bar{\xi}$-regular and $\bar{\xi}$-irregular points }\label{srir}

From Theorem $\ref{t1}$, we know that for any $s\in(0,\infty)$,
\beqnn
P_{\bar{\xi}}\left(\displaystyle\lim_{n}Z_{n}(\bar{\xi})
h_{n}(\bar{\xi},s)=\infty\right)>0,~~P_{\bar{\xi}}\left(\displaystyle\lim_{n}Z_{n}(\bar{\xi})
h_{n}(\bar{\xi},s)=0\right)>0~~~\eta\text{-a.e.,}
\eeqnn
then it is necessary to distinguish between two types of points from the following definition.

\begin{definition}\label{rir}
A point $s\in(0,\infty)$ is called $\bar{\xi}$-regular if $P_{\bar{\xi}}\left(W(\bar{\xi},s)\in\{0,\infty\}\right)=1$,
 and $\bar{\xi}$-irregular otherwise.
\end{definition}

We can get the following theorem which gives a necessary and sufficient condition for a point to be regular. The proof is almost the same as that of Theorem 1.1.2 in $\cite{SB77}$, we omit the details.
\begin{theorem}\label{re1}
$s\in(0,\infty)$ is $\bar{\xi}$-regular if and only if $\displaystyle\lim_{n\to\infty}\frac{h_{n}(\bar{\xi},t)}{h_{n}(\bar{\xi},s)}=0$ for all $0<t<s,($or equivalently $\displaystyle\lim_{n\to\infty}\frac{h_{n}(\bar{\xi},t)}{h_{n}(\bar{\xi},s)}=\infty$ for all $s<t<\infty).$
\end{theorem}
\qed

\begin{remark}\label{rm00}

\noindent From Theorem $\ref{re1}$ and the fact that
\beqnn
h_{n}(\bar{\xi},t)=h_{n-k}\left(\theta^{k}\bar{\xi},h_{k}(\bar{\xi},t)\right),
\eeqnn
we know that
if $s\in(0,\infty)$ is $\bar{\xi}$-regular (irregular), then $h_{k}(\bar{\xi},s)$ is $\theta^{k}\bar{\xi}$-regular (irregular).\qed

\end{remark}

\begin{lemma}\label{open}
The set of the irregular points is open. More precisely, if $s_{i}$ is $\bar{\xi}$-irregular, then an open interval $I(\bar{\xi},s_{i})=(s_{1},s_{2})$ of maximal length exists, such that $s_{i}\in I(\bar{\xi},s_{i})$ and all
$s\in I(\bar{\xi},s_{i})$ are $\bar{\xi}$-irregular.

If we define $l(\bar{\xi},s)=\lim_{n\to\infty}\left(h_{n}(\bar{\xi},s)/h_{n}(\bar{\xi},s_{i})\right)$ for all $s\in[s_{1},s_{2}]$ then $l(\bar{\xi},s)$ is a continuous, strictly increasing function with $l(\bar{\xi},s_{1})=0$ and $l(\bar{\xi},s_{2})=\infty$. $s_{1}$ and $s_{2}$ are both $\bar{\xi}$-regular.
\end{lemma}
The proof is similar as Lemma 1.1.5 in $\cite{SB77}$, we omit the details.
\qed

\begin{lemma}\label{er} Under condition (A2) (i.e., for any $\bar{\xi}$ that satisfies for any $s\in(0,\infty),~d(\bar{\xi},s)=0$), every interval $[h_{\xi_{0}}(s),s],~0<s<\infty$, contains at least one $\theta\bar{\xi}$-regular point $s_{r}$.
\end{lemma}
\begin{proof}
Define
\beqnn
s_{r}:=\sup\left\{t \Big|~h_{\xi_{0}}(s)\leqslant t\leqslant s ~\text{and} ~\lim_{n}\frac{h_{n}(\theta\bar{\xi},t)}{h_{n}(\theta\bar{\xi},s)}=0\right\},
\eeqnn
$d(\bar{\xi},s)=0$ tells us
\beqnn
d(\bar{\xi},s)=\displaystyle\lim_{n}\frac{h_{n+1}(\bar{\xi},s)}{h_{n}(\theta\bar{\xi},s)}=
\lim_{n}\frac{h_{n}\left(\theta\bar{\xi},h_{\xi_{0}}(s)\right)}{h_{n}(\theta\bar{\xi},s)}=0.
\eeqnn
 Then $s_{r}$ exists and $h_{\xi_{0}}(s)\leqslant s_{r}\leqslant s.$

If $\displaystyle\lim_{n\to\infty}\frac{h_{n}(\theta\bar{\xi},s_{r})}{h_{n}(\theta\bar{\xi},s)}=0$, then
$\displaystyle\lim_{n\to\infty}\frac{h_{n}(\theta\bar{\xi},s_{r})}{h_{n}(\theta\bar{\xi},t)}=0$ for all $t>s_{r}$, since $\displaystyle\lim_{n\to\infty}\frac{h_{n}(\theta\bar{\xi},t)}{h_{n}(\theta\bar{\xi},s)}>0$; and if $\displaystyle\lim_{n\to\infty}\frac{h_{n}(\theta\bar{\xi},s_{r})}{h_{n}(\theta\bar{\xi},s)}>0$, then
$\displaystyle\lim_{n\to\infty}\frac{h_{n}(\theta\bar{\xi},t)}{h_{n}(\theta\bar{\xi},s_{r})}=0$ for all $t<s_{r}$, since $\displaystyle\lim_{n\to\infty}\frac{h_{n}(\theta\bar{\xi},t)}{h_{n}(\theta\bar{\xi},s)}=0$.
In both cases $s_{r}$ is $\theta\bar{\xi}$-regular by Theorem $\ref{re1}$.
\qed
\end{proof}

\begin{definition}\label{rp}
The branching process $\{Z_{n}(\bar{\xi})\}$ is called $\bar{\xi}$-regular if all $0<s<\infty$ are $\bar{\xi}$-regular and $\bar{\xi}$-irregular otherwise.
\end{definition}

We have the following 0-1 law.

\begin{theorem}
Let $A=\{\bar{\xi}:Z_{n}(\bar{\xi})~\text{is}~\bar{\xi}\text{-regular}\}$, then $\eta(A)=0~\text{or}~1$.
\end{theorem}

\begin{proof}
From Definition \ref{rp} and Theorem $\ref{re1}$ we know that for any $\bar{\xi}\in A$, $0<s<\infty$, if $0<t<s$, then
\beqnn
\lim_{n\to\infty}\frac{h_{n}(\bar{\xi},t)}{h_{n}(\bar{\xi},s)}=0.
\eeqnn
On the other hand, for any $\bar{\xi}\in A$, if $\{Z_{n}(\theta\bar{\xi})\}$ is $\theta\bar{\xi}$-irregular, then there exists $s,~0<t<s,$ such that $\displaystyle\lim_{n\to\infty}\frac{h_{n}(\theta\bar{\xi},t)}{h_{n}(\theta\bar{\xi},s)}>0,i.e.,$
\begin{eqnarray}\label{y0}
\lim_{n\to\infty}\frac{h_{\xi_{n}}\left(\cdots h_{\xi_{1}}(h_{\xi_{0}}(k_{\xi_{0}}(t)))\cdots\right)}{h_{\xi_{n}}\left(\cdots h_{\xi_{1}}(h_{\xi_{0}}(k_{\xi_{0}}(s)))\cdots\right)}>0.
\end{eqnarray}
Combining with the monotonicity of $k_{\xi_{0}}$, we have $k_{\xi_{0}}(t)<k_{\xi_{0}}(s)$.
$(\ref{y0})$ means $k_{\xi_{0}}(s)$ is a $\bar{\xi}$-irregular point, as a consequence $\{Z_{n}(\bar{\xi})\}$ is $\bar{\xi}$-irregular, which contradicts to the fact that $\bar{\xi}\in A$. Thus, for any $\bar{\xi}\in A,~\theta\bar{\xi}\in A$. In a similar way we see that for any $\theta\bar{\xi}\in A,~\bar{\xi}\in A$. So $
\theta A=A,
$
i.e., A is a $\theta$-invariant set. Since $\theta$ is ergodic, we infer that $
\eta(A)=0~~\text{or}~~1.$
\qed
\end{proof}

\noindent Thus we can make the following definition classifying the processes.

\begin{definition}\label{rpr}
The branching process in random environment $\{Z_{n}\}$ is called regular branching process if $\eta\left(\{\bar{\xi}:Z_{n}(\bar{\xi})~\text{is}~ \bar{\xi}\text{-regular}\}\right)=1$, otherwise irregular.
\end{definition}

The following results can also be proved  similar as that of Theorem 1.1.7 in $\cite{SB77}$, we omit the details.

\begin{theorem}\label{ca1}
\noindent $\left( \text{1} \right)$ Let $c_{n}(\bar{\xi})$ be a sequence of positive constants, such that $Z_{n}(\bar{\xi})/c_{n}(\bar{\xi})$ converges in distribution, and let $F_{\bar{\xi}}$ denote the distribution function of the limit. Then there are four cases:
\begin{alignat*}{2}
(a)&~F_{\bar{\xi}}(0)=1 \quad\quad\quad&~~~  \Longrightarrow \quad\quad\quad\quad   &\lim_{n} h_{n}(\bar{\xi},s)c_{n}(\bar{\xi})=\infty~~\text{for all }~~ 0<s<\infty;\\
(b)&~F_{\bar{\xi}}(0)=F_{\bar{\xi}}(\infty)=0 \quad\quad\quad&~~~ \Longrightarrow \quad\quad\quad\quad   &\lim_{n} h_{n}(\bar{\xi},s)c_{n}(\bar{\xi})=0~~\text{for all }~~ 0<s<\infty;\\
(c)&~1>F_{\bar{\xi}}(0)=F_{\bar{\xi}}(\infty)>0 \quad\quad\quad&~~~ \Longrightarrow \quad\quad\quad\quad   &\text{a $\bar{\xi}$-regular point $s_{r}$ exists such that}\\
&&&\lim_{n}h_{n}(\bar{\xi},t)c_{n}(\bar{\xi})=
\begin{cases}
0 & \text{if~  $0<t<s_{r}$}\\
\infty & \text{if~  $s_{r}<t<\infty$};
\end{cases}\\
(d)&~F_{\bar{\xi}}(0)<F_{\bar{\xi}}(\infty) \quad\quad\quad&~~~ \Longrightarrow \quad\quad\quad\quad   &\text{a $\bar{\xi}$-irregular point $s_{i}$ exists such that}\\
&&&\lim_{n}h_{n}(\bar{\xi},s_{i})c_{n}(\bar{\xi})=1.
\end{alignat*}

\noindent $\left( \text{2} \right)$ On the other hand, in (1) if one of the conditions on the right-hand side is satisfied, then $Z_{n}(\bar{\xi})/c_{n}(\bar{\xi})$ converges almost surely to a (possibly infinite-valued) random variable $W(\bar{\xi})$; more specifically there are four cases:
\begin{alignat*}{2}
(a)&~\lim_{n} h_{n}(\bar{\xi},s)c_{n}(\bar{\xi})=\infty~~~\text{for all }~~ 0<s<\infty& \Longrightarrow \quad  &\lim_{n}\frac{Z_{n}}{c_{n}}=0~~P_{\bar{\xi}}\text{-a.s.};\\
(b)&~\lim_{n} h_{n}(\bar{\xi},s)c_{n}(\bar{\xi})=0~~~\text{for all }~~ 0<s<\infty \quad\quad&~~ \Longrightarrow \quad  &\lim_{n}\frac{Z_{n}}{c_{n}}=\lim_{n}Z_{n}(\bar{\xi})~~P_{\bar{\xi}}\text{-a.s.};\\
(c)&~\text{a $\bar{\xi}$-regular point $s_{r}$ exists such that}\\
&~\lim_{n}h_{n}(\bar{\xi},t)c_{n}(\bar{\xi})=
\begin{cases}
0 & \text{if~  $0<t<s_{r}$}\\
\infty & \text{if~  $s_{r}<t<\infty$}
\end{cases} \quad\quad&~~ \Longrightarrow \quad  &\lim_{n}\frac{Z_{n}}{c_{n}}=W(\bar{\xi},s_{r})~~P_{\bar{\xi}}\text{-a.s.};
\\
(d)&~\text{a $\bar{\xi}$-irregular point $s_{i}$ exists such that}\\
&~\lim_{n}h_{n}(\bar{\xi},s_{i})c_{n}(\bar{\xi})=1\quad\quad&~~ \Longrightarrow \quad  &\lim_{n}\frac{Z_{n}}{c_{n}}=W(\bar{\xi},s_{i})~~P_{\bar{\xi}}\text{-a.s.}.
\end{alignat*}
\end{theorem}

\qed

\begin{remark} Combining with Theorem \ref{t1},
Theorem $\ref{ca1}$ implies that in our case, for a.e. $\bar{\xi}$, no $c_{n}(\bar{\xi})$ exists that $Z_{n}(\bar{\xi})/c_{n}(\bar{\xi})$ has a proper and non-degenerate limit. Morever, if $\{Z_{n}\}$ is a regular branching process, then  the growth of $Z_{n}$  can not be measured by a sequence of positive constants.
\end{remark}

\subsection{Normalized by a sequence of functions}\label{ms3}

Since for $\eta$-a.e. $\bar{\xi}$, $\lim_{n}Z_{n}(\bar{\xi})/c_{n}(\bar{\xi})$ is never a proper, non-degenerate random variable, we now consider other possibilities for normalizing $Z_{n}(\bar{\xi})$.

For $0\leqslant x<\infty$, let
\beqnn
y_{n}\left(\bar{\xi},x\right)=f_{\xi_{0}}\left(\cdots\left(f_{\xi_{n-1}}\left(e^{-\frac{1}{x}}\right)\right)\cdots\right), ~y_{n}(\bar{\xi},0)=f_{n}(\bar{\xi},0).
\eeqnn

\begin{theorem}\label{yu1}
(1) $y_{n}\left(\bar{\xi},Z_{n}(\bar{\xi})\right)$ converges almost surely to a random variable $Y(\bar{\xi})$. Furthermore, if $s_{r}$ is a $\bar{\xi}$-regular point and $x_{r}=e^{-s_{r}}$, then $P_{\bar{\xi}}\left(Y(\bar{\xi})\leqslant x_{r}\right)=x_{r}$ and $P_{\bar{\xi}}\left(Y(\bar{\xi})= x_{r}\right)=0$.

\noindent(2) $Y\in(0,1)~\eta$-a.s..  In particular, if $\{Z_{n}\}$ is a regular branching process, then $Y$ is uniformly distributed on $(0,1)$.

\end{theorem}

\begin{remark}
The proof of the first part is similar to that of Theorem 2.1.1 in $\cite{SB77}$. But to conclude that $Y\in(0,1)~\eta$-a.s. can not followed there, because
if $s\in(0,\infty)$ is $\bar{\xi}$-regular (irregular), we only know that $h_{k}(\bar{\xi},s)$ is  $\theta^{k}\bar{\xi}$-regular (irregular) from Remark \ref{rm00}.

\end{remark}

\begin{proof} (1) The proof of the first part is similar to the discussion of Theorem 2.1.1 in $\cite{SB77}$, we rewrite it briefly as follows.
For any $x\in(0,1)$,
\begin{eqnarray}\label{y1}
\begin{split}
\left\{y_{n}(\bar{\xi},Z_{n}(\bar{\xi}))<x~~\text{eventually}\right\}&=\left\{Z_{n}(\bar{\xi})<\left(-\log f_{n}^{(-1)}(\bar{\xi},x)\right)^{-1}~~\text{eventually}\right\}\\
& \supseteq \left\{W(\bar{\xi},-\log x)<1\right\},
\end{split}
\end{eqnarray}
and
\begin{eqnarray}\label{y2}
\begin{split}
\left\{y_{n}(\bar{\xi},Z_{n}(\bar{\xi}))>x~~\text{eventually}\right\}&=\left\{Z_{n}(\bar{\xi})>\left(-\log f_{n}^{(-1)}(\bar{\xi},x)\right)^{-1}~~\text{eventually}\right\}\\
& \supseteq \left\{W(\bar{\xi},-\log x)>1\right\}.
\end{split}
\end{eqnarray}

From $(\ref{y1})~,(\ref{y2})$, similar to the discussion of Theorem 2.1.1 in $\cite{SB77}$, we have
$y_{n}(\bar{\xi},Z_{n}(\bar{\xi}))$ converges almost surely to a random variable $Y(\bar{\xi})$.

If $s_{r}=-\log x_{r}$ is a $\bar{\xi}$-regular point, from the property of regular points and $(\ref{y1})~,(\ref{y2})$ we have
\beqnn
P_{\bar{\xi}}\left(Y(\bar{\xi})<e^{-s_{r}}\right)\geqslant P_{\bar{\xi}}\left(W(\bar{\xi},s_{r})<1\right)&=&P_{\bar{\xi}}\left(W(\bar{\xi},s_{r})=0\right)
=E_{\bar{\xi}}\left(X(\bar{\xi},s_{r})\right)=e^{-s_{r}},\\
P_{\bar{\xi}}\left(Y(\bar{\xi})>e^{-s_{r}}\right)\geqslant P_{\bar{\xi}}\left(W(\bar{\xi},s_{r})>1\right)&=&P_{\bar{\xi}}\left(W(\bar{\xi},s_{r})=\infty\right)
=1-e^{-s_{r}}.
\eeqnn
Thus
\beqnn
P_{\bar{\xi}}\left(Y(\bar{\xi})=x_{r}\right)=0~~~\text{and}~~~P_{\bar{\xi}}\left(Y(\bar{\xi})\leqslant x_{r}\right)=x_{r}.
\eeqnn

(2)We only need to prove $\mathbb{P}(Y=1)=0,~\mathbb{P}(Y=0)=0.$
Let
\beqnn
D=\left\{\bar{\xi}:\text{for any }s\in(0,\infty),~d(\bar{\xi},s)=0\right\},
\eeqnn
from Lemma $\ref{er}$ we know that for any $\bar{\xi}\in D,~s\in(0,\infty)$, there exists at least one $\theta\bar{\xi}$-regular point in $[h_{\xi_{0}}(s),s].$ If we assume that
\beqnn
P_{\bar{\xi}}\left(Y(\bar{\xi})=1\right)=\delta>0,
\eeqnn
we claim that for any $s\in \left(0,-\log (1-\delta)\right)$, $s$ is a $\bar{\xi}$-irregular point. Otherwise, if there exists $s\in \left(0,-\log (1-\delta)\right)$ a $\bar{\xi}$-regular point, then
\beqnn
P_{\bar{\xi}}(Y(\bar{\xi})<e^{-s})=e^{-s}>1-\delta,
\eeqnn
that contradicts to
$P_{\bar{\xi}}(Y(\bar{\xi})=1)=\delta>0.$

Then for any
\beqnn
s\in \left(0,h_{\xi_{0}}(-\log (1-\delta))\right),
\eeqnn
$s$ is a
$\theta\bar{\xi}$-irregular point. But since $\bar{\xi}\in D$, we already know (Lemma $\ref{er}$) that for any
\beqnn
0<s_{0}<h_{\xi_{0}}\left(-\log (1-\delta)\right),
\eeqnn
there exists at least one $\theta\bar{\xi}$-regular point in $[h_{\xi_{0}}(s_{0}),s_{0}]$, since
\beqnn
[h_{\xi_{0}}(s_{0}),s_{0}]\subseteq \left(0,h_{\xi_{0}}\left(-\log (1-\delta)\right)\right),
\eeqnn
thus the assumption is not valid, i.e.,
\beqnn
P_{\bar{\xi}}\left(Y(\bar{\xi})=1\right)=0~~~~~\text{for any }\bar{\xi}\in D.
\eeqnn

In a similar way we can prove that
\beqnn
P_{\bar{\xi}}\left(Y(\bar{\xi})=0\right)=0~~~~~\text{for any }\bar{\xi}\in D.
\eeqnn
Accordingly $\mathbb{P}\left(Y\in(0,1)\right)=1$ because under the assumption (A2), $\mathbb{P}(D)=1$.

In particular, if $\{Z_{n}\}$ is a regular branching process, for $\eta$-a.e. $\bar{\xi}$, any $e^{-s}\in(0,1)$, $s$ is a $\bar{\xi}$-regular point, $P_{\bar{\xi}}\left(Y(\bar{\xi})\leqslant e^{-s}\right)=e^{-s},$ obviously, $Y$ is uniformly distributed on $(0,1).$
\qed
\end{proof}

We define the random variable
\beqnn
T(\bar{\xi})=\sup\left\{s|0<s<\infty~~\text{and}~~~W(\bar{\xi},s)<1\right\}.
\eeqnn
Then for any $\bar{\xi} \in D~\left( D=\left\{\bar{\xi}:\text{for any } s\in(0,\infty),~d(\bar{\xi},s)=0\right\}\right),$
\beqnn
1\geqslant P_{\bar{\xi}}\left(W(\bar{\xi},s)=0\right)\geqslant P_{\bar{\xi}}\left(W\left(\bar{\xi},k_{\xi_{0}}(s)\right)<\infty\right)\geqslant
e^{-k_{\xi_{0}}(s)}\stackrel{s\to 0} {\longrightarrow} 1,
\eeqnn
and
\beqnn
1\geqslant P_{\bar{\xi}}\left(W(\bar{\xi},s)=\infty\right)\geqslant P_{\bar{\xi}}\left(W(\bar{\xi},h_{\xi_{0}}(s))>0\right)\geqslant
1-e^{-h_{\xi_{0}}(s)}\stackrel{s\to \infty} {\longrightarrow} 1.
\eeqnn
This implies that for any $\bar{\xi}\in D$, $P_{\bar{\xi}}$-a.s. $W\left(\bar{\xi},s\right)=0$ if $s$ is close to 0 and
 $W\left(\bar{\xi},s\right)=\infty$ if $s$ is large enough. Therefore by Lemma $\ref{open}$ for $\bar{\xi}\in D$, either
\begin{eqnarray}\label{T1}
T(\bar{\xi}) \text{ is a }\bar{\xi}\text{-regular point with } W\left(\bar{\xi},s\right)=
\begin{cases}
0~~~\text{if   } 0<s<T(\bar{\xi})\\
\infty~~\text{if } T\left(\bar{\xi}\right)<s<\infty,
\end{cases}
\end{eqnarray}
or
\begin{eqnarray}\label{T2}
T\left(\bar{\xi}\right) \text{ is a }\bar{\xi}\text{-irregular point with }W\left(\bar{\xi},T(\bar{\xi})\right)=1~~\text{and } W\left(\bar{\xi},s\right)
\begin{cases}
<1~~\text{for } s<T(\bar{\xi})\\
>1~~\text{for } s>T(\bar{\xi}).
\end{cases}
\end{eqnarray}

\begin{corollary}\label{cor1}
Under Assumption \ref{Ass}, for $\eta$-a.s.$~\bar{\xi}$,
\begin{eqnarray}\label{w1}
W\left(\bar{\xi},s\right)<1<W\left(\bar{\xi},t\right)~~~\text{for   }~0<s<T(\bar{\xi})<t<\infty.
\end{eqnarray}
\begin{eqnarray}
T(\bar{\xi})=-\log Y(\bar{\xi}),
\end{eqnarray}
and therefore $T(\bar{\xi})\in(0,\infty)$, for any $\bar{\xi}$-regular point $s_{r}$,
\begin{eqnarray}
P_{\bar{\xi}}\left(T(\bar{\xi})\geqslant s_{r}\right)=e^{-s_{r}}~~~\text{and}~~~
P_{\bar{\xi}}\left(T(\bar{\xi})= s_{r}\right)=0.
\end{eqnarray}
\end{corollary}

\begin{proof}
$(\ref{w1})$ follows immediately from $(\ref{T1})$ and $(\ref{T2})$. From $(\ref{y1})$ and $(\ref{y2})$ we have
\beqnn
e^{-T(\bar{\xi})}=\inf\left\{x|W(\bar{\xi},-\log x)<1\right\}\geqslant Y(\bar{\xi})\geqslant\sup\left\{x|W(\bar{\xi},-\log x)>1\right\},
\eeqnn
and by $(\ref{w1})$
\beqnn
\inf\left\{x|W(\bar{\xi},-\log x)<1\right\}=\sup\left\{x|W(\bar{\xi},-\log x)>1\right\},
\eeqnn
hence, $T(\bar{\xi})=-\log Y(\bar{\xi})$. Other properties are easy corollaries from Theorem $\ref{yu}$.
\qed
\end{proof}

\subsection{Proof of Theorem $\ref{t2}$}\label{ms4}

In order to prove Theorem $\ref{t2}$, we need the following two lemmas, which are similar as Lemma 2.2.4 and Lemma 2.2.5 in $\cite{SB77}$, we omit the details of the proof.

\begin{lemma}\label{sim}
Let $U(\bar{\xi})$ be as in Theorem $\ref{t2}$ and
\beqnn
V(\bar{\xi},x):=\inf\left\{y| y\geqslant 0 \text{ and } U(\bar{\xi},y)\geqslant x\right\},~~0\leqslant x<\infty.
\eeqnn
Then
\begin{eqnarray}
U\left(\bar{\xi},V(\bar{\xi},x)\right)\sim x
\end{eqnarray}
\end{lemma}
\qed

\begin{lemma}\label{case}
Let $F_{\bar{\xi}}$ be as in Theorem $\ref{t2}$ (2). If $F_{\bar{\xi}}$ is continuous at $x\in(0,\infty),$ then either $F_{\bar{\xi}}(x)=0$ and
\beqnn
\lim_{n} V\left(\bar{\xi},c_{n}(\bar{\xi})x\right)h_{n}(\bar{\xi},t)=0~~~~~\text{for  } 0<t<\infty,
\eeqnn
or $F_{\bar{\xi}}(x)=1$
\beqnn
\lim_{n} V\left(\bar{\xi},c_{n}(\bar{\xi})x\right)h_{n}(\bar{\xi},t)=\infty~~~~~\text{for  } 0<t<\infty,
\eeqnn
or $s:=-\log F_{\bar{\xi}}(x)$ is a $\bar{\xi}$-regular point and
\beqnn
\lim_{n} V\left(\bar{\xi},c_{n}(\bar{\xi})x\right)h_{n}(\bar{\xi},t)=
\begin{cases} 0 ~~~ \text{  if  }~~0<t<s\\
\infty ~~ \text{  if  }~~s<t<\infty.
\end{cases}
\eeqnn
\qed
\end{lemma}

\begin{proof}\textbf{of Theorem~\ref{t2}}
(1) If $s_{i}$ is a $\bar{\xi}$-irregular point, and if $H(\bar{\xi},s_{i})=\lim_{n}\left(U\left(\bar{\xi},\frac{1}{h_{n}(\bar{\xi},s_{i})}\right)/c_{n}(\bar{\xi})\right)$ exists, then by Lemma $\ref{open}$, the limit $H(\bar{\xi},s)$ exists for all $I(\bar{\xi},s_{i})$ and is equal to $H
(\bar{\xi},s_{i})$, since $U(\bar{\xi})$ varies slowly. Thus $H(\bar{\xi})$ is continuous at $s_{i}$.

Since now we assume that $(\ref{exi})$ holds, then the points where the limit $H(\bar{\xi})$ does not exist can only be $\bar{\xi}$-regular points, and there are at most countably many such points. Further $H(\bar{\xi})$ is a monotonic function, and therefore at most countably many points in $(0,\infty)$ exist, where $H(\bar{\xi})$ is not continuous. Since $\bar{\xi}\in D$, by Corollary $\ref{cor1}$, for any $\bar{\xi}$-regular point $s_{r}$, $P_{\bar{\xi}}(T(\bar{\xi})= s_{r})=0$, and $T(\bar{\xi})\in(0,\infty)$ $P_{\bar{\xi}}$-a.s..

Take $0<s<T(\bar{\xi})<t<\infty$, then by $(\ref{w1})$,
\beqnn
1/h_{n}(\bar{\xi},t)<Z_{n}(\bar{\xi})<1/h_{n}(\bar{\xi},s)\quad\quad\text{eventually.}
\eeqnn
Therefore,
\beqnn
U\left(\bar{\xi},1/h_{n}(\bar{\xi},t)\right)/c_{n}(\bar{\xi})\leqslant
U\left(\bar{\xi},Z_{n}(\bar{\xi})\right)/c_{n}(\bar{\xi})\leqslant U\left(\bar{\xi},1/h_{n}(\bar{\xi},s)\right)/c_{n}(\bar{\xi})
\quad\text{eventually.}
\eeqnn
Since $H(\bar{\xi})$ is continuous at $T(\bar{\xi})$, and $H(\bar{\xi})$ exists at s and t arbitrarily close to $T(\bar{\xi})$, we have
\begin{eqnarray}
\lim_{n}U\left(\bar{\xi},Z_{n}(\bar{\xi})\right)/c_{n}(\bar{\xi})=H\left(\bar{\xi},T(\bar{\xi})\right)~~\text{almost surely}.
\end{eqnarray}

(2) We define
\beqnn
G\left(\bar{\xi},x\right)=\inf\left\{y~|~0\leqslant y<\infty~\text{and}~F_{\bar{\xi}}(y)\geqslant x\right\},~~~0\leqslant x<\infty.
\eeqnn
Let $G(\bar{\xi})$ be continuous at $y\in(0,1)$, and $x=G\left(\bar{\xi},y\right)$. Since $F_{\bar{\xi}}$ is right-continuous as a distribution function, we always have
\begin{eqnarray}\label{barf}
F_{\bar{\xi}}\left(G(\bar{\xi},y)\right)=F_{\bar{\xi}}\left(\inf\{z|F_{\bar{\xi}}(z)\geqslant y\}\right)\geqslant y.
\end{eqnarray}

If $F_{\bar{\xi}}(x)=y$, then $F_{\bar{\xi}}$ is strictly increasing at $x$, since $G(\bar{\xi})$ is continuous at $y$. We can choose $x_{1}<x<x_{2}$ arbitrarily close to $x$, such that $F_{\bar{\xi}}$ is continuous at $x_{1}$ and $x_{2}$. Lemma $\ref{case}$ implies
\begin{eqnarray}\label{bijiao}
V\left(\bar{\xi},c_{n}(\bar{\xi})x_{1}\right)<\left(1/h_{n}(\bar{\xi},s)\right)<V\left(\bar{\xi},c_{n}(\bar{\xi})x_{2}\right)
\end{eqnarray}
eventually for $s=-\log F_{\bar{\xi}}(x)=-\log y$.

Therefore by Lemma $\ref{sim}$, when $n\to\infty$,
\begin{eqnarray*}
\frac{U\left(\bar{\xi},V(\bar{\xi},c_{n}(\bar{\xi})x_{1})\right)}{c_{n}(\bar{\xi})}&\longrightarrow& x_{1},\\
\frac{U\left(\bar{\xi},V(\bar{\xi},c_{n}(\bar{\xi})x_{2})\right)}{c_{n}(\bar{\xi})}&\longrightarrow& x_{2}.
\end{eqnarray*}
Thus,
\beqnn
\lim_{n}U\left(\bar{\xi},1/h_{n}(\bar{\xi},s)\right)/c_{n}\left(\bar{\xi}\right)=x=G\left(\bar{\xi},e^{-s}\right)
\eeqnn
exists.

If on the other hand $F_{\bar{\xi}}(x)>y$, then
\beqnn
a:=F_{\bar{\xi}}(x-)<y<F_{\bar{\xi}}(x)=:b.
\eeqnn
Choose $x_{1}<x<x_{2}$ arbitrarily close to $x$ such that $F_{\bar{\xi}}$ is continuous at $x_{1}$ and $x_{2}$, and hence
\beqnn
-\log F_{\bar{\xi}}(x_{2})\leqslant -\log b<s=-\log y<-\log a\leqslant-\log F_{\bar{\xi}}(x_{1}).
\eeqnn
Again Lemma $\ref{case}$ implies $(\ref{bijiao})$ and therefore
$\displaystyle\lim_{n}U\left(\bar{\xi},1/h_{n}(\bar{\xi},s)\right)/c_{n}\left(\bar{\xi}\right)=G\left(\bar{\xi},e^{-s}\right).$

(3) Since both $G(\bar{\xi})$ and $G(\theta\bar{\xi})$ are continuous at all but at most countably many $s\in(0,\infty)$, there exists at least one $s_{0}\in(0,\infty)$ that $G(\bar{\xi})$ and $G(\theta\bar{\xi})$ are both continuous at $e^{-s_{0}},~e^{-h_{\xi_{0}}(s_{0})}.$
Then from $(\ref{2.3})$, for $\eta$-a.s. $\bar{\xi}$,
\begin{eqnarray}\label{ubar1}
\lim_{n\to\infty}U\left(\bar{\xi},\frac{1}{h_{n}(\bar{\xi},s_{0})}\right)/c_{n}\left(\bar{\xi}\right)=G\left(\bar{\xi},e^{-s_{0}}\right)~~~\text{exists},
\end{eqnarray}
\begin{eqnarray}\label{ubar2}
\lim_{n\to\infty}U\left(\theta\bar{\xi},\frac{1}{h_{n-1}(\theta\bar{\xi},h_{\xi_{0}}(s_{0}))}\right)/c_{n-1}\left(\theta\bar{\xi}\right)
=G\left(\theta\bar{\xi},e^{-h_{\xi_{0}}(s_{0})}\right)~~~\text{exists}.
\end{eqnarray}
Since from our assumption we know that $U\left(\bar{\xi},\frac{1}{h_{n}\left(\bar{\xi},s_{0}\right)}\right)=U\left(\theta\bar{\xi},\frac{1}{h_{n-1}\left(\theta\bar{\xi},h_{\xi_{0}}(s_{0})\right)}\right)$,
moreover, assumptions on $F_{\bar{\xi}}$ ensures that for $0<y<1$, $0<G(\bar{\xi},y)<\infty$ for $\eta$-a.s.$~\bar{\xi}$. Then combine with $(\ref{ubar1})$, $(\ref{ubar2})$, we have
\begin{eqnarray}\label{ubar3}
\lim_{n\to\infty}\frac{c_{n-1}(\theta\bar{\xi})}{c_{n}(\bar{\xi})}=
\frac{G\left(\bar{\xi},e^{-s_{0}}\right)}{G\left(\theta\bar{\xi},e^{-h_{\xi_{0}}(s_{0})}\right)}:=\alpha(\bar{\xi})>0.
\end{eqnarray}

Then $(\ref{2.3})$, $(\ref{ubar3})$ and the assumption $U(\bar{\xi})=U(\theta\bar{\xi})$ imply
\begin{eqnarray}\label{conti}
\frac{G\left(\bar{\xi},e^{-s}\right)}{G\left(\theta\bar{\xi},e^{-h_{\xi_{0}}(s)}\right)}=
\lim_{n\to\infty}\frac{c_{n-1}\left(\theta\bar{\xi}\right)}{c_{n}\left(\bar{\xi}\right)}=\alpha\left(\bar{\xi}\right)
\end{eqnarray}
for all s for which $G(\bar{\xi})$ is continuous at $e^{-s}$. In particular $G(\bar{\xi})$ is continuous at $e^{-s}$ if and only if $G(\theta\bar{\xi})$ is continuous at $e^{-h_{\xi_{0}}(s)}$.

Since $G(\bar{\xi})$ is left-continuous, $(\ref{conti})$ is true for all $s\in(0,\infty)$.

Now for every $0<u<\infty$,
\begin{eqnarray}\label{yf}
\left\{y|G(\theta\bar{\xi},y)\leqslant u\right\}=\left\{y|y\leqslant F_{\theta\bar{\xi}}(u)\right\},
\end{eqnarray}
since $F_{\theta\bar{\xi}}(x)<y$ for all $x<G\left(\theta\bar{\xi},y\right)$ by definition, and because of $(\ref{barf})$.

Our assumption ensures that $-\log F_{\theta\bar{\xi}}(u)\in(0,\infty)$ for $0<u<\infty$.
Since $F_{\theta\bar{\xi}}$ is right-continuous there exists a sequence of points $u_{n}>u$, such that
$\lim_{n\to\infty} F_{\theta\bar{\xi}}(u_{n})=F_{\theta\bar{\xi}}(u)$, and $F_{\theta\bar{\xi}}$ is continuous at every $u_{n}$. Lemma $\ref{case}$ implies that $-\log F_{\theta\bar{\xi}}(u_{n})$ are
$\theta\bar{\xi}$-regular points, and therefore for every $0<u<\infty$,
\beqnn
-\log F_{\theta\bar{\xi}}(u) \text{ is a $\theta\bar{\xi}$-regular point,}
\eeqnn
since $\{s|s \text{ is } \theta\bar{\xi}\text{-regular}\}\cup \{0,\infty\}$ is a closed set by Lemma $\ref{open}$.

Now since
\beqnn
k_{\xi_{0}}\left(-\log F_{\theta\bar{\xi}}(u)\right)=-\log f_{\xi_{0}}\left(F_{\theta\bar{\xi}}(u)\right),
\eeqnn
$-\log f_{\xi_{0}}\left(F_{\theta\bar{\xi}}(u)\right)\text{ is a $\bar{\xi}$-regular point}.$

Theorem $\ref{t2}$ (1),(2) tells us that $\lim_{n\to\infty}U\left(\bar{\xi},Z_{n}(\bar{\xi})\right)/c_{n}\left(\bar{\xi}\right)=G\left(\bar{\xi},Y(\bar{\xi})\right)~~P_{\bar{\xi}}$-a.s..
This combines with $(\ref{gth})$, $(\ref{yf})$ imply
\beqnn
F_{\bar{\xi}}\left(\alpha(\bar{\xi})u\right)&=&P_{\bar{\xi}}\left(G\left(\bar{\xi},Y(\bar{\xi})\right)\leqslant \alpha(\bar{\xi})u\right)
=P_{\bar{\xi}}\left(G\left(\bar{\xi},e^{-T(\bar{\xi})}\right)\leqslant \frac{G\left(\bar{\xi},e^{-T(\bar{\xi})}\right)}{G\left(\theta\bar{\xi},e^{-h_{\xi_{0}}(T(\bar{\xi}))}\right)} u\right )\\
&=& P_{\bar{\xi}}\left(G\left(\theta\bar{\xi},e^{-h_{\xi_{0}}(T(\bar{\xi}))}\leqslant u\right)\right)
=P_{\bar{\xi}}\left(e^{-h_{\xi_{0}}(T(\bar{\xi}))}\leqslant F_{\theta\bar{\xi}}(u)\right)\\
&=&P_{\bar{\xi}}\left(f_{\xi_{0}}^{(-1)}\left(e^{-T(\bar{\xi})}\right)\leqslant F_{\theta\bar{\xi}}(u)\right)
=P_{\bar{\xi}}\left(Y\left(\bar{\xi}\right)\leqslant f_{\xi_{0}}\left(F_{\theta\bar{\xi}}(u)\right)\right)\\
&=&f_{\xi_{0}}\left(F_{\theta\bar{\xi}}(u)\right)
\eeqnn
for $0<u<\infty$, the last equality is due to the fact that $-\log f_{\xi_{0}}(F_{\theta\bar{\xi}}(u))$ is a $\bar{\xi}$-regular point and Theorem \ref{yu}. Since
\beqnn
F_{\bar{\xi}}\left(\alpha(\bar{\xi})\cdot 0\right)=F_{\bar{\xi}}(0)=0=f_{\xi_{0}}\left(F_{\theta\bar{\xi}}(0)\right),
\eeqnn
$(\ref{xidi})$ is true.
\qed
\end{proof}

\begin{remark}
Since $T\left(\bar{\xi}\right)\in(0,\infty)$, Theorem $\ref{t2}$ $(1)$ tells us that if we can find suitable $U(\bar{\xi},x)$, $c_{n}(\bar{\xi})$ that makes $H\left(\bar{\xi},s\right):=\lim_{n}U\left(\bar{\xi},1/h_{n}(\bar{\xi},s)\right)/c_{n}\left(\bar{\xi}\right)$ exists for all but at most countably many $s\in(0,\infty)$ and $0<H\left(\bar{\xi},s\right)<\infty$ for $s\in(0,\infty)$, then $U\left(\bar{\xi},Z_{n}(\bar{\xi})\right)/c_{n}\left(\bar{\xi}\right)$ has a non-degenerate and proper limit. What's more, $(1)$ combines with $(2)$ show that if $U\left(\bar{\xi},Z_{n}(\bar{\xi})\right)/c_{n}\left(\bar{\xi}\right)$ converges in distribution, then it must converges almost surely.
\end{remark}

\section{Sufficient  criteria for regular process}\label{sms}

From Definition $\ref{rpr}$ we know that for a regular process, for $a.e.~\bar{\xi}$, any $s<t$, $\lim_{n\to\infty}\frac{h_{n}\left(\bar{\xi},s\right)}{h_{n}\left(\bar{\xi},t\right)}=0$, then $d\left(\bar{\xi},s\right)=0$ since
\beqnn
d(\bar{\xi},s)=\lim_{n\to\infty}\frac{h_{n}\left(\bar{\xi},s\right)}{h_{n-1}\left(\theta\bar{\xi},s\right)}=
\lim_{n\to\infty}\frac{h_{n}\left(\bar{\xi},s\right)}{h_{n}\left(\bar{\xi},k_{\xi_{0}}(s)\right)},
\eeqnn
where $k_{\xi_{0}}(s)>s$. Thus, every regular process satisfies Assumption {(\bf A2)}. In this section, we will derive some sufficient conditions for a process to be regular.

Let $Q_{\xi_{i}}:~[0,1)\longrightarrow[0,1)$ defined by
\beqnn
Q_{\xi_{i}}(s)=\frac{f_{\xi_{i}}^{'}(s)(1-s)}{1-f_{\xi_{i}}(s)}.
\eeqnn
Since $f_{\xi_{i}}(x)$ is strictly convex and $f_{\xi_{i}}(1)=1$, then $f_{\xi_{i}}^{'}(s)<\frac{1-f_{\xi_{i}}(s)}{1-s}$, i.e.
\beqnn
0\leqslant Q_{\xi_{i}}(s)<1.
\eeqnn

\begin{theorem}\label{zxy}
$s\in(0,\infty)$ is $\bar{\xi}$-regular if and only if $\displaystyle\prod_{n=0}^{\infty} Q_{\xi_{n}}\left(f_{n+1}^{(-1)}\left(\bar{\xi},e^{-s}\right)\right)=0,$ where $f_{n}(\bar{\xi},s)=f_{\xi_{0}}\left(\cdots(f_{\xi_{n-1}}(s))\cdots\right)$.
\end{theorem}

\begin{proof}
From the proof of Theorem $\ref{re1}$(for details see \cite{SB77} Theorem 1.1.2) we know that $s$ is $\bar{\xi}$-regular if and only if
\beqnn
\lim_{n} k_{n}\left(\bar{\xi},h_{n}(\bar{\xi},s)x\right)=s~~~\text{for all } 0<x<\infty.
\eeqnn
Since $k_{n}\left(\bar{\xi},h_{n}(\bar{\xi},s)x\right)$ is a concave function of $x$, $k_{n}\left(\bar{\xi},h_{n}(\bar{\xi},s)\cdot 1\right)=s$ and $k_{n}\left(\bar{\xi},h_{n}(\bar{\xi},s)x\right)\leqslant s$ for all $x\leqslant 1(\geqslant s~\text{for all } x\geqslant 1)$, this is equivalent to
\begin{eqnarray}\label{reeq}
\gamma_{n}\left(\bar{\xi},s\right)=\frac{d}{dx}\big(k_{n}(\bar{\xi},h_{n}(\bar{\xi},s)\cdot x)\big)\bigg|_{x=1}\stackrel{n\to\infty} {\longrightarrow}0.
\end{eqnarray}
We can calculate that
\beqnn
\gamma_{n}\left(\bar{\xi},s\right)=e^{s}\cdot f_{n}^{'}\left(\bar{\xi},f_{n}^{(-1)}\left(\bar{\xi},e^{-s}\right)\right)f_{n}^{(-1)}\left(\bar{\xi},e^{-s}\right)\left(-\log f_{n}^{(-1)}\left(\bar{\xi},e^{-s}\right)\right),
\eeqnn
and since
\beqnn
\lim_{n\to\infty}\left(-\log f_{n}^{(-1)}\left(\bar{\xi},e^{-s}\right)\right)\left(1- f_{n}^{(-1)}\left(\bar{\xi},e^{-s}\right)\right)\to 1,
\eeqnn
$(\ref{reeq})$ is equivalent to
\beqnn
f_{n}^{'}\left(\bar{\xi},f_{n}^{(-1)}\left(\bar{\xi},e^{-s}\right)\right)\left(1- f_{n}^{(-1)}\left(\bar{\xi},e^{-s}\right)\right)\to 0.
\eeqnn
Since
\beqnn
f_{n}^{'}\left(\bar{\xi},f_{n}^{(-1)}\left(\bar{\xi},e^{-s}\right)\right)=\prod_{j=0}^{n-1}f_{\xi_{j}}^{'}\left(f_{j+1}^{(-1)}\left(\bar{\xi},e^{-s}\right)\right),
\eeqnn
thus $(\ref{reeq})$ is equivalent to $\displaystyle\prod_{n=0}^{\infty} Q_{\xi_{n}}\left(f_{n+1}^{(-1)}\left(\bar{\xi},e^{-s}\right)\right)=0.$
\qed
\end{proof}

\begin{corollary}\label{coreg}
If $\mathbb P\left(\left\{\xi_{0}:\sup_{s<1}Q_{\xi_{0}}(s)\leqslant c<1\right\}\right)>0$, then $\{Z_{n}\}$ is a regular branching process.
\end{corollary}

\begin{proof}
\begin{eqnarray}\label{ifp2}
\mathbb E\prod_{i=0}^{\infty}Q_{\xi_{i}}\left(f_{i+1}^{(-1)}(e^{-s})\right)&=&
\mathbb Ee^{\log \prod_{i=0}^{\infty}Q_{\xi_{i}}\left(f_{i+1}^{(-1)}(e^{-s})\right)}\nonumber \\
&=&\mathbb Ee^{\sum_{i=0}^{\infty}\log Q_{\xi_{i}}\left(f_{i+1}^{(-1)}(e^{-s})\right)}\nonumber\\
&\leqslant& \mathbb Ee^{\sum_{i=0}^{\infty}\log\left(\sup_{0<s<1} Q_{\xi_{i}}(s)\right)}.
\end{eqnarray}
If
\begin{eqnarray}\label{ifp}
\mathbb P\left(\left\{\xi_{0}:\sup_{s<1}Q_{\xi_{0}}(s)\leqslant c<1\right\}\right)>0,
\end{eqnarray}
since $\{\xi_{i}\}$ is a sequence of independent and identically distributed random variables, we know that
 $\mathbb P$-a.e.,
\beqnn
\frac{\sum_{i=0}^{n}\log\left(\sup_{0<s<1} Q_{\xi_{i}}(s)\right)}{n}\stackrel{n\to\infty} {\longrightarrow}
\mathbb E\log\left(\sup_{0<s<1} Q_{\xi_{0}}(s)\right).
\eeqnn
Since $\mathbb P(\sup Q_{\xi_{0}}(s)\leqslant 1)=1$, $(\ref{ifp})$ ensures that on a set with positive probability $\sup_{0<s<1} Q_{\xi_{0}}(s)<1$, thus $\mathbb E \log(\sup_{0<s<1} Q_{\xi_{i}}(s))<0$, that is to say, for $\eta$-a.e. $\bar{\xi}$,
\beqnn
\sum_{i=0}^{\infty}\log \left(\sup_{0<s<1} Q_{\xi_{i}}(s)\right)=-\infty~~~P_{\bar{\xi}}\text{-a.e.}.
\eeqnn
As a result,
\beqnn
\mathbb E\prod_{i=0}^{\infty}Q_{\xi_{i}}\left(f_{i+1}^{(-1)}(e^{-s})\right)=0,
\eeqnn
for any $0<s<\infty$.

From Theorem $\ref{zxy}$ we know that for $\eta$-a.e.$~\bar{\xi}$, every $s\in(0,\infty)$ is a $\bar{\xi}$-regular point, then $\{Z_{n}(\bar{\xi})\}$ is $\bar{\xi}$ regular $\eta$-a.e., and $\{Z_{n}\}$ is a regular branching process by Definition
$\ref{rpr}$.
\qed
\end{proof}

\begin{example}\label{exm}
Let $f_{\xi_{i}}(s)=1-(1-s)^{\alpha_{\xi_{i}}}$, where $\{\alpha_{\xi_{i}}\}$ is a collection of independent and identically distributed random variables, taking values in $(0,1-\epsilon)$ ($0<\epsilon<1$ is a constant). Then this process is a regular process which satisfies our assumption and $U(\bar{\xi},x)=\log x,~c_{n}(\bar{\xi})=\frac{1}{\alpha_{\xi_{0}}}\cdots\frac{1}{\alpha_{\xi_{n-1}}}$ is a suitable choice for the normalization of $\{Z_{n}\}$.
\end{example}
\begin{proof}
$\left( 1 \right)$ Since $f_{\xi_{0}}(s)=1-(1-s)^{\alpha_{\xi_{0}}}$, we can calculate that $f'_{\xi_{0}}(s)=\alpha_{\xi_{0}}(1-s)^{\alpha_{\xi_{0}}-1}$. Since $0<\alpha_{\xi_{0}}<1-\epsilon, \mathbb{E}\log m(\xi_{0})=\infty$.

Let $r_{n}(\bar{\xi},s):=1-f_{\xi_{n-1}}^{(-1)}\left(\cdots \left(f_{\xi_{0}}^{(-1)}(1-s)\right)\cdots\right)$ for all $0\leqslant  s\leqslant 1$, then it is clear that for all $0<s<1$,
\beqnn
r_{n}(\bar{\xi},s)\sim h_{n}\left(\bar{\xi},-\log (1-s)\right).
\eeqnn
In our example, it is easy to calculate that $r_{n}(\bar{\xi},s)=s^{\frac{1}{\alpha_{\xi_{0}}}\frac{1}{\alpha_{\xi_{1}}}\cdots\frac{1}{\alpha_{\xi_{n-1}}}}$. Then for any $0<s<1$,
\beqnn
d\left(\bar{\xi},-\log (1-s)\right)=\lim_{n\to\infty}\frac{h_{n+1}\left(\bar{\xi},-\log (1-s)\right)}{h_{n}\left(\theta\bar{\xi},-\log (1-s)\right)}=\lim_{n\to\infty}\frac{r_{n+1}\left(\bar{\xi},s\right)}{r_{n}\left(\theta\bar{\xi},s\right)}=
\lim_{n\to\infty}\frac{s^{\frac{1}{\alpha_{\xi_{0}}}\frac{1}{\alpha_{\xi_{1}}}\cdots\frac{1}{\alpha_{\xi_{n}}}}}
{s^{\frac{1}{\alpha_{\xi_{1}}}\cdots\frac{1}{\alpha_{\xi_{n}}}}}=0,
\eeqnn
$\eta$-a.e.$~\bar{\xi}$, that means for any $0<s<\infty$, $d(\bar{\xi},s)=0$. Combined with the fact that $f_{\xi_{0}}(0)=0$, this example satisfies Assumption ${\bf(A1)}$ and ${\bf(A2)}$.

$\left( 2 \right)$ Since $Q_{\xi_{i}}(s)=\frac{f'_{\xi_{i}}(s)(1-s)}{1-f_{\xi_{i}}(s)}=\alpha_{\xi_{i}}$, $\mathbb{P}\left(\left\{\xi_{0}:\sup_{s<1} Q_{\xi_{0}}(s)\leqslant 1-\epsilon\right\}\right)=1$. From Corollary \ref{coreg} we know that $\{Z_{n}\}$ is a regular branching process.

$\left( 3 \right)$ If we choose
\beqnn
U(\bar{\xi},x)=\log x,~~c_{n}(\bar{\xi})=\frac{1}{\alpha_{\xi_{0}}}\cdots\frac{1}{\alpha_{\xi_{n-1}}},
\eeqnn
by calculation,
\beqnn
\lim_{n\to\infty}U\left(\bar{\xi},\frac{1}{h_{n}(\bar{\xi},s)}\right)/c_{n}(\bar{\xi})&=&
\lim_{n\to\infty}U\left(\bar{\xi},\frac{1}{r_{n}(\bar{\xi},1-e^{-s})}\right)/c_{n}(\bar{\xi})\\
&=&\frac{\log (1-e^{-s})^{-\frac{1}{\alpha_{\xi_{0}}}\cdots\frac{1}{\alpha_{\xi_{n-1}}}}}{\frac{1}{\alpha_{\xi_{0}}}\cdots\frac{1}{\alpha_{\xi_{n-1}}}}\\
&=&-\log (1-e^{-s})\in (0,\infty).
\eeqnn
Then from Theorem $\ref{t2}$ we know that in this case
\beqnn
\lim_{n\to\infty}\frac{\log\left( Z_{n}(\bar{\xi})\right)}{c_{n}\left(\bar{\xi}\right)}=-\log \left(1-e^{-T(\bar{\xi})}\right)=-\log \left(1-Y(\bar{\xi})\right)\in(0,\infty).
\eeqnn
Thus this is a suitable choice for the normalization of $\{Z_{n}\}.$

$\left( 4 \right)$ Note that in this example $\lim_{n \to \infty} \frac{c_{n-1}\left(\theta\bar{\xi}\right)}{c_{n}\left(\bar{\xi}\right)}=\alpha_{\xi_{0}}$. If we use $F_{\bar{\xi}}$ to denote the distribution function of the limit of $\frac{\log \left(Z_{n}(\bar{\xi})\right)}{c_{n}\left(\bar{\xi}\right)}$, we have
\beqnn
F_{\bar{\xi}}(x)=P_{\bar{\xi}}\left(-\log (1-Y(\bar{\xi}))\leqslant x\right).
\eeqnn
Since $\{Z_{n}\}$ is a regular branching process, Theorem \ref{yu} tells us that for $\eta$-a.s. $\bar{\xi}$, $Y(\bar{\xi})$ is uniformly distributed on $(0,1)$. Thus $F_{\bar{\xi}}(x)=1-e^{-x}$. Then
\beqnn
f_{\xi_{0}}\left(F_{\theta\bar{\xi}}(u)\right)=f_{\xi_{0}}\left(1-e^{-u}\right)=1-e^{-\alpha_{\xi_{0}}u}=F_{\bar{\xi}}(\alpha_{0}u),
\eeqnn
which coincides with (\ref{xidi}).
\qed
\end{proof}


\begin{thebibliography}{99}
\baselineskip=14pt

\bibitem{A13} Amini, O., Devroye, L., Griffiths, S.,  Olver, N. (2013) On explosions in heavy-tailed branching random walks. {\it Ann. Prob}. \textbf{41(3B)}, 1864-1899.



\bibitem{AK711} Athreya, K.B. and Kailin, S. (1971) On branching processes with random environments I: Extinction probabilities. {\it Ann. Math. Statist}. \textbf{42}, 1499-1520.

\bibitem{AK712} Athreya, K.B. and Kailin, S. (1971) On branching processes with random environments II: Limit theorems. {\it Ann. Math. Statist}. \textbf{42}, 1843-1858.

\bibitem{AN72} Athreya, K.B. and Ney, P.E. (1972) Branching Processes.{\it Springer, Berlin}.

\bibitem{CH77} Cohn, H. (1977) Almost sure convergence of branching processes. {\it Z. Wahrscheinlichkeitsth}. \textbf{38}, 73-81.

\bibitem{DD70} Darling, D.A. (1970) The Galton-Wstson process with infinite mean. {\it J. Appl. Prob}. \textbf{7}, 455-456.

\bibitem{GD77} Grey, D.R. (1977) Almost sure convergence in Markov branching processes with infinite mean. {\it J. Appl. Prob}. \textbf{14}, 702-716.

\bibitem{HC70} Heyde, C.C. (1970) Extension of a result of Seneta for the supercritical Galton-Watson process. {\it Ann. Math. Statist}. \textbf{41}, 739-742.

\bibitem{KS66} Kesten, H. and Stigum, B.P. (1966) A limit theorem for multidimensional Galton-Watson processes. {\it Ann. Math. Statist}. \textbf{37}, 1211-1223.

\bibitem{SB77} Schuh, H.-J. and Barbour, A.D. (1977) On the asymptotic behaviour of branching processes with infinite mean. {\it Adv. Appl. Prob}. \textbf{9}, 681-723.

\bibitem{SE69} Seneta, E. (1969) Functional equations and the Galton-Watson process. {\it Adv. Appl. Prob}. \textbf{1}, 1-42.

\bibitem{SE73} Seneta, E. (1973) The simple branching process with infinite mean I. {\it J. Appl. Prob}. \textbf{10}, 206-212.

\bibitem{TD77} Tanny, D. (1977) Limit theorems for branching processes in a random environment. {\it Ann. Prob}. \textbf{5}, 100-116.

\bibitem{TD78} Tanny, D. (1978) Normalizing constants for branching processes in random environments (B.P.R.E.). {\it Stoch. Proc. and Their Appl}. \textbf{6}, 201-211.

\bibitem{TD88} Tanny, D. (1988) A necessary and sufficient condition for a branching process in a random environment to grow like the product of its means. {\it Stoch. Proc. and Their Appl}. \textbf{28}, 123-139.

\bibitem{V87} Vatutin, V. A. (1987) Sufficient conditions for the regularity of bellman-harris branching processes. {\it Theory Probab. Appl}. \textbf{31}, 50-57.   

 \end{thebibliography}
\end{document}